\documentclass[a4paper,10pt]{amsart}

\usepackage{amsmath}
\usepackage{amsfonts}
\usepackage{amssymb}
\usepackage{import}
\usepackage[all]{xy}

\renewcommand{\AA}{\mathbb{A}}

\newcommand{\affin}{\mathbb{A}}

\newcommand{\C}{\mathcal{C}}
\newcommand{\D}{\mathcal{D}}

\newcommand{\sset}{\mathrm{s}\mathcal{S}\mathrm{et}}
\newcommand{\smk}{\mathcal{S}\mathrm{m}/k}

\newcommand{\set}{\mathcal{S}\mathrm{et}}
\newcommand{\Set}{\set}				
\newcommand{\sSet}{\sset}

\newcommand{\op}{{\mathrm{op}}}

\newcommand{\pr}{\mathrm{pr}}
\newcommand{\ADelta}{\mbox{{$\Delta$ \hspace{-2.37ex}{\fontsize{7pt}{0pt}$\Delta$}}}}

\newcommand{\Sing}{\mathrm{Sing}}

\newcommand{\Spec}{\mathrm{Spec}\hspace{2pt}}

\newcommand{\adjoint}{~\rightleftarrows~}

\newcommand{\ssim}{\hspace{-3.2pt}\mathrel{\vcenter{\offinterlineskip\vskip+5.5pt\hbox{$\sim$}}}}
\newcommand{\upperminus}{{\hspace{-1pt}\mbox{\fontsize{5}{5}\selectfont$($-$)$}}}

\DeclareMathOperator{\Hom}{\hom}

\DeclareMathOperator{\id}{id}

\DeclareMathOperator*{\colim}{colim}

\DeclareMathOperator{\Pre}{Pre}
\DeclareMathOperator{\sPre}{sPre}

\newcommand{\ssPre}{\mathcal{S}\mathrm{Pre}}
\newcommand{\SPre}{\ssPre}

\theoremstyle{plain} 
\newtheorem{theorem}{Theorem}[section]
\newtheorem{corollary}[theorem]{Corollary}
\newtheorem{lemma}[theorem]{Lemma}
\theoremstyle{definition}
\newtheorem{defi}[theorem]{Definition}

\newtheorem{remark}[theorem]{Remark}
\numberwithin{equation}{section}

\input xy 				
\xyoption{all} 		

\title{Enriched simplicial presheaves and the Motivic Homotopy Category}
\author{Philip Herrmann \and Florian Strunk}

\begin{document}

\begin{abstract}
We construct models for the motivic homotopy category based on simplicial functors from smooth schemes over a field to simplicial sets. These spaces are homotopy invariant and therefore one does not have to invert the affine line in order to get a model for the motivic homotopy category.
\end{abstract}

\maketitle

\section*{Introduction}
In this note, we study certain simplicial functors as an alternative for simplicial presheaves in the construction of the motivic homotopy category. An enriched simplicial presheaf is a simplicial functor from a category of schemes enriched over simplicial sets to the category of simplicial sets enriched over itself. Considering enriched simplicial presheaves instead of simplicial presheaves seems to be quite natural in the spirit of motivic homotopy theory. For example there is a naive homotopy contracting the affine line in the category of schemes. More precisely, for any constant map $c$ there exists a morphism $H$ of smooth schemes over a field, such that the diagram
\begin{equation*}
\begin{xy}
\xymatrix{
\affin^1\ar@{->}[dr]_{i_0}\ar@{->}@/^/[drr]^{c}&\\
&\affin^1\times\affin^1\ar@{->}[r]^(0.57){H}&\affin^1\\
\affin^1\ar@{->}@/_/[urr]_{\id}\ar@{->}[ur]^{i_1}&
}
\end{xy}
\end{equation*}
commutes. The simplicial presheaf represented by $\affin^1$ resists to be weakly equivalent to the point until it is finally forced to be weakly contractible by Bousfield localization. In contrast to this the enriched simplicial presheaf represented by $\affin^1$ is objectwise contractible (cf. Corollary \ref{a1contr}). Hence the motivic models based on these spaces can be obtained without the $\affin^1$-contracting Bousfield localization. 

\subsubsection*{Acknowledgements}
We would like to thank Prof. Oliver R\"ondigs for his support and many helpful discussions.

\subsubsection*{Conventions}
Throughout this text let $k$ be a field and $\smk$ the category of smooth and separated schemes of finite type over $k$. The category of simplicial ($\Set$-valued) presheaves on $\smk$ is denoted by $\sPre$.

\section{The category of enriched simplicial presheaves}
In this section we introduce the category $\ssPre$ of \emph{enriched simplicial presheaves} as an alternative
for the category $\sPre$ of simplicial presheaves. The construction of $\ssPre$ is based on categories enriched over simplicial sets. In a simplicial category $\C$ there are $\Hom$-simplicial sets $\sset_\C(A,B)$ instead of just $\Hom$-sets associated with any two objects, in a way compatible with an associative and unital composition. The $0$-simplices of $\sset_\C(A,B)$ can be thought of as morphisms $A\to B$. The relation of being connected by a zig-zag of $1$-simplices models a notation of \emph{naive homotopy} depending on the enrichment. In the following we consider the category $\sset$ of simplicial sets as a simplicial category by
\[ \sset_{\sset}(A,B)_n = \Hom_{\sset}(A\times\Delta^n,B).\]
The naive homotopy relation turns out to be pretty sensible in the sense that it coincides with a notation of left homotopy in the usual model structure on simplicial sets. This enrichment is natural in many aspects, for example it is given by the Yoneda embedding and the following straightforward lemma.

\begin{lemma}
 Let $\C$ be a category with finite products. Any cosimplicial object $c:\Delta\to\C$ with $c_0$ the 
terminal object of $\C$ gives rise to a simplicial category, which we also denote by $\C$, with underlying category $\C$ and
\[ \sset_\C(A,B)_n= \Hom_\C(A\times c_n,B).\]
\end{lemma}
\begin{proof}
A map $\sigma:[m]\to[n]$ in $\Delta$ induces a map
$\sset_\C(A,B)_n\to\sset_\C(A,B)_m$ by assigning the composite
\[A\times c([m])
\xrightarrow{(\pr_1,c(\sigma)\circ\pr_2)}A\times c([n]) \xrightarrow{f} B
\]
to $f\in\sset_\C(A,B)_n$. Clearly $\sset_\C(A,B)(\id_{[n]})=\id_{\sset_\C(A,B)_n}$ and one observes that for composable morphisms $\sigma$ and $\tau$ in $\Delta$ the identity
\begin{eqnarray*}
\sset_\C(A,B)(\tau\circ\sigma)(f) &=& f\circ\left( \pr_1,
\left(c(\tau\circ\sigma)\circ\pr_2\right)\right)\\
&=& \sset_\C(A,B)(\sigma)\circ \sset_\C(A,B)(\tau)(f)
\end{eqnarray*}
holds and hence $\sset_\C(A,B)$ is in fact a simplicial set. The composition maps
$$
c_{ABC}:\sset_\C(B,C)\times\sset_\C(A,B)\to\sset_\C(A,B), \quad(g,f)\mapsto g\circ(f,\pr_2)
$$
are maps of simplicial sets and satisfy the relevant coherence diagrams \cite[6.9,6.10]{borc942}. The underlying category $U\C$ has by definition the same objects as $\C$ and the $\Hom$-sets are given by
\begin{eqnarray*}
 \Hom_{U\C}(A,B)&:=&\Hom_{\sset}(\Delta[0],\sset_\C(A,B))\\
&\cong& \sset_\C(A,B)_0\\
&\cong& \Hom_\C(A, B).
\end{eqnarray*}
The composition in $U\C$ is the same as composition in simplicial dimenstion $0$ of the enriched category and therefore $U\C\cong \C$.
\end{proof}

By applying this lemma to the \emph{algebraic cosimplicial object} $\ADelta^\upperminus$ given by
$$
\textstyle
	\ADelta^p=\Spec k[X_0,\ldots,X_p]/{\left(1-\sum X_i\right)}
$$
one obtains $\smk$ as a simplicial category.

\begin{defi}
The category $\ssPre$ of \emph{enriched simplicial presheaves} is the category of simplicial functors from $\smk^\op$ to $\sset$, i.e.~functors $X$ assigning a simplicial set $XU$ to any smooth $k$-scheme $U$ and a morphism
\[ \sset_{\smk}(U,V)\to\sset_{\sset}(XV,XU)\]
of simplicial sets to any pair of objects $U,V$ compatible with composition.
\end{defi}

\begin{remark}
The notation of naive homotopy in the simplicial category $\smk$ is not completly convenient, but includes some reasonable aspects as for example
\[\sset_{\smk}(S^1_t,S^1_t)_*/{\hspace{2pt}\ssim_{\mbox{\tiny naive}}}\]
equals the integers. A discussion of this naive homotopy relation in $\smk$ can be found in section 2 of \cite{morel2004motivic}.
\end{remark}

\begin{lemma}[Adjunction Lemma]\label{adjunctionlemma}
Let $\D$ be an essentially small category, $\C$ a cocomplete category and $c:\D\to\C$ a functor. There exists a commutative diagram
$$
\xymatrix{
	\D\ar[r]^(0.35){\mbox{\tiny Yoneda}}\ar[dr]_c	&	\Pre(\D)\ar[d]^{|-|}\\
						&\C
}
$$
and an adjunction $|\hspace{-2pt}-\hspace{-2pt}|:\Pre(\D)\adjoint \C: \Sing$ with $\Sing(X)=\hom(c(-),X)$.
\end{lemma}
\begin{proof}
 This is a standard fact about left Kan extensions \cite{borc94}.
\end{proof}

The Adjunction Lemma \ref{adjunctionlemma} applied to the functor \[c:\smk\times\Delta\to\ssPre,\quad (U,[n])\mapsto \sset_{\smk}(-,U)\times\Delta^n\]
provides an adjunction
\begin{equation}\label{rladj}
L:\sPre\rightleftarrows\ssPre:R.
\end{equation}
The composite functor $RL$ is well known and was already studied in \cite{mv99} as a functor called $\Sing$, defined by \[\Sing(X)(U)_m=\Hom_{\Pre}(U\times\ADelta^m,X_m).\]

\begin{lemma}\label{rlsing}The functors $RL$ and $\Sing$ coincide.
\end{lemma}
\begin{proof}
Since the functors $R,L$ and $\Sing$ preserve colimits we only need to check their behavior on representable objects.
\begin{eqnarray*}
  RL(U\times\Delta^n)(V,[m]) &=& \Hom_{\ssPre}(\sset_{\smk}(-,V)\times\Delta^m,\sset_{\smk}(-,U)\times\Delta^n)\\
&\cong&\sset_{\ssPre}(\sset_{\smk}(-,V),\sset_{\smk}(-,U)\times\Delta^n)_m\\
&\cong&\Hom_{\smk}(V\times\ADelta^m,U)\times\Delta^n_m\\
&\cong&U(V\times\ADelta^m)_m\times\Delta^n_m\\
&\cong&\Hom_{\Pre}(V\times \ADelta^m,U_m)\times\Delta^n_m\\
&\cong& \Sing(U\times\Delta^n)(V)_m
\end{eqnarray*}
\end{proof}

\begin{corollary}\label{a1contr}
The enriched simplicial presheaf represented by the affine line is objectwise contractible. 
\end{corollary}
\begin{proof}
As a corollary of Lemma \ref{rlsing} we obtain
\begin{eqnarray*}
\affin^1(U)&=&\sset_{\smk}(U,\affin^1)=L\affin^1(U)\\
&=&RL\affin^1(U)=\Sing(\affin^1)(U)
\end{eqnarray*}
 which is contractible by \cite[Corollary 3.5]{mv99}.
\end{proof}

\begin{lemma}\label{bicomplete}
 The category of enriched simplicial presheaves is bicomplete and colimits and limits can be computed objectwise.
\end{lemma}
\begin{proof}
 The category $\ssPre$ is the underlying category of a $\sset$-category in which all weighted $\sset$-colimits and limits exist \cite[Proposition 6.6.17]{borc942}, so $\ssPre$ is bicomplete by \cite[Proposition 6.6.16]{borc942}.
\end{proof}

We use the conventional terminology and say that a set $I$ of morphisms in a category \emph{permits the small object argument}, if the domains of the elements of $I$ are small relative to transfinite compositions of pushouts of elements in $I$.

\begin{lemma}\label{allsmall} Let $I$ be a set of morphisms in $\sPre$. Then the set $LI$ of morphisms in $\ssPre$ permits the small object argument.
\end{lemma}
\begin{proof}
We make use of the fact that all objects in the locally presentable category $\sPre$ are small. So there exists a cardinal $\kappa$, such that for all $\kappa$-filtered ordinals $\lambda$ and any $\lambda$-sequence $S:\lambda\to\ssPre$ the following diagram commutes. 
\begin{equation*}
\begin{xy}
\xymatrix{ 
\colim\limits_{\beta<\lambda}\Hom_{\ssPre}(LX,F_\beta)\ar@{->}[d]_{\cong}\ar@<3.8pt>@{->}[r]^{\Phi}&\Hom_{\ssPre}(LX,\colim\limits_{\beta<\lambda}F_\beta)\ar@{->}[d]^{\cong}\\
\colim\limits_{\beta<\lambda}\Hom_{\sPre}(X,RF_\beta)\ar@<3.8pt>@{->}[r]^{\cong}&\Hom_{\sPre}(X,\colim\limits_{\beta<\lambda}RF_\beta)
}
\end{xy}
\end{equation*}
Hence $LX$ is small and $LI$ permits the small object argument.
\end{proof}

\section{Model structures for enriched simplicial presheaves}

In this section we construct model structures on the category $\SPre$ of enriched simplicial presheaves. These model structures correspond to model structures on the category $\sPre$ of simplicial presheaves. Subsequently, Corollary \ref{fibrant} gives a characterization of the fibrant objects.

\begin{defi}
Let $\C$ and $\D$ be a model categories and $L:\C\adjoint \D:R$ an adjunction. The model structure on $\D$ is called \emph{$R$-lifted} if a morphism $f$ of $\D$ is a weak equivalence (resp.~a fibration) if and only if $R(f)$ is a weak equivalence (resp.~a fibration) of $\C$. A cofibrantly generated model category $\C$ is called \emph{$(I,J)$-cofibrantly generated} if $I$ is a set of generating cofibrations and $J$ is a set of generating acyclic cofibrations for the model structure on $\C$. 
\end{defi}

\begin{remark}
 If $\C$ is a model category, $L:\C\adjoint \D:R$ an adjunction and $\D$ is equipped with the $R$-lifted model structure, then the adjunction $(L,R)$ is necessarily a Quillen adjunction since the right adjoint $R$ preserves fibrations and acyclic fibrations. The lifted model structure on $\D$ is right proper if and only if $\C$ is a right proper model category.
\end{remark}

\begin{lemma}[Lifting Lemma]\label{liftinglemma}
Let $\C$ be a $(I,J)$-cofibrantly generated model category, $\D$ a bicomplete category and $L:\C\adjoint \D:R$ an adjunction such that the right adjoint $R$ commutes with colimits and $LI$ and $LJ$ permit the small object argument. Then there exists a unique $(LI,LJ)$-cofibrantly generated $R$-lifted model structure on $\D$ if and only if for every $j\in J$ and every pushout diagram
$$
\xymatrix{
L(A)\ar[r]\ar[d]_{L(j)}	&	X\ar[d]^p	\\
L(B)\ar[r]		&	Y
}
$$
the morphism $R(p)$ is a weak equivalence of $\C$.
\end{lemma}
\begin{proof}
This is a standard lifting argument \cite[Theorem 11.3.2]{hirsch}.
\end{proof}

\begin{theorem}\label{liftingtheorem}
Consider the adjunction
$$
	L:\sPre\adjoint \SPre:R
$$
constructed in \eqref{rladj}. Let $\sPre$ be equipped with a cofibrantly generated model structure with $\AA^1$-local weak equivalences as weak equivalences and with the property that every cofibration is in particular a monomorphism. Then the $R$-lifted model structure on $\SPre$ exists and the adjunction $(L,R)$ is a Quillen equivalence.
\end{theorem}
\begin{proof}
Let $I$ be a set of generating cofibrations and $J$ be a set of generating acyclic cofibrations for the model structure on $\sPre$, $j$ an element of $J$ and 
$$
\xymatrix{
L(A)\ar[r]\ar[d]_{L(j)}	&	X\ar[d]^p	\\
L(B)\ar[r]		&	Y
}
$$
be a pushout diagram in $\SPre$. Since $R$ commutes with colimits, the diagram
$$
\xymatrix{
RL(A)\ar[r]\ar[d]_{RL(j)}	&	R(X)\ar[d]^{R(p)}	\\
RL(B)\ar[r]			&	R(Y)
}
$$
is also a pushout. The morphism $j$ is an acyclic cofibration of $\sPre$ and therefore in particular an acyclic cofibration in the $\AA^1$-local injective model structure on $\sPre$, that is a $\AA^1$-local weak equivalence and a monomorphism. Lemma \ref{rlsing} identifies the functor $RL$ with the singular functor $\Sing$. The singular functor respects monomorphisms and $\AA^1$-local weak equivalences by \cite[Corollary 3.8]{mv99}. Therefore $RL(j)$ is an acyclic cofibration in the $\AA^1$-injective model structure on $\sPre$. The class of acyclic cofibrations of a model category is closed under pushouts and hence $R(p)$ is a $\AA^1$-local weak equivalence. The category $\SPre$ is bicomplete by Lemma \ref{bicomplete} and Lemma \ref{allsmall} provides that $LI$ and $LJ$ permit the small object argument. Hence the category $\SPre$ can be equipped with the $R$-lifted model structure by Lemma \ref{liftinglemma}. To prove that $(L,R)$ is a Quillen equivalence, let $\eta$ be the unit of the adjunction $(L,R)$ and let $X$ be a simplicial presheaf. Lemma \ref{rlsing} identifies $\eta(X)$ with the canonical morphism $X\to \Sing(X)$ which is a $\AA^1$-local weak equivalence by \cite[Corollary 3.8]{mv99}. The diagram
$$
\xymatrix{
X\ar[r]^(0.38){\eta(X)}\ar[rd]_(0.48){f^\sharp}	&	RL(X)\ar[d]^{R(f)}	\\
					&	R(Y)
}
$$
shows that a morphism $f:LX\to Y$ is a weak equivalence if and only if its adjoint $f^\sharp$ is a weak equivalence. Therefore  $(L,R)$ is a Quillen equivalence.
\end{proof}

\begin{remark}
The assumptions on the model structure on $\sPre$ of Theorem \ref{liftingtheorem} are fulfilled by all intermediate model structures, e.g. the projective, flasque and injective model structures.
\end{remark}

\begin{lemma}\label{verymonoidal}
Consider the adjunction $L:\sPre\adjoint \SPre:R$ and let $(\sPre,\times)$ be equipped with a monoidal model structure. If the category $(\SPre,\times)$ is endowed with the $R$-lifted model structure, then it is a monoidal model category.
\end{lemma}
\begin{proof}
General results on enriched category theory imply that $\SPre$ is cartesian closed \cite{day}. Let $i:A\to B$ and $j:C\to D$ be cofibrations. One has to show that the \emph{pushout product}
\[\textstyle
i\hspace{1pt}{\mbox{\Tiny$\square$}}\hspace{1pt} j: (B\times C) \coprod_{(A\times C)}(A\times D) \to B\times D
\]
is a cofibration and an acyclic cofibration if $i$ or $j$ is a weak equivalence. This follows from the property of $L$ being a left Quillen functor and from the relation $L(i\hspace{1pt}{\mbox{\Tiny$\square$}}\hspace{1pt} j)\cong L(i)\hspace{1pt}{\mbox{\Tiny$\square$}}\hspace{1pt} L(j)$ holding as the functor $L$ is strong monoidal, which is the case since
$$
\begin{array}{rcl}
L(X\times Y)	&=& L(\colim(\hom(-,U)\times\Delta^n)\times \colim(\hom(-,V)\times\Delta^m))\\
		&=& L(\colim(\hom(-,U\times V)\times\Delta^n\times\Delta^m))\\
		&=& \colim(\sSet(-,U\times V)\times\Delta^n\times\Delta^m)\\
		&=& \colim(\sSet(-,U)\times\Delta^n)\times \colim(\sSet(-,V)\times\Delta^m)\\
		&=& L(X)\times L(Y).	        
\end{array}
$$ 
\end{proof}

\begin{lemma}
Consider the adjunction $L:\sPre\adjoint \SPre:R$ and let $\sPre$ be equipped with a simplicial model structure. If the category of enriched simplicial presheaves is endowed with the $R$-lifted model structure, then it is a simplicial model category.
\end{lemma}
\begin{proof}
The category $\SPre$ is naturally enriched over the category of simplicial sets by $\sSet(X,Y)=\hom_{\SPre}(X\times\Delta^\upperminus,Y)$. It is tensored with $X\otimes A=X(-)\times A$ and cotensored with $X^A=\hom_{\sSet}(A\times\Delta^\upperminus,X(-))$. By Lemma \ref{verymonoidal} a statement equivalent to the (SM7) axiom holds \cite[II.3.11]{goerss}.
\end{proof}

\begin{lemma}\label{homotopyinvariant}
Every enriched simplicial presheaf $X$ is homotopy invariant, that is the map
$$
	X(U)\to X(U\times\AA^1)
$$
induced by the projection is a weak equivalence of $\sSet$ for all objects $U$ of $\smk$.
\end{lemma}
\begin{proof}
An enriched simplicial presheaf $X$ maps a morphism $f:U\to V$ of $\smk$ to a $0$-simplex of the simplicial set $\sSet(XV,XU)$ and it maps a naive homotopy $H:U\times\ADelta^1\to V$ of $\smk$ to a $1$-simplex of $\sSet(XV,XU)$, which is a homotopy equivalence of the simplicial sets $XV$ and $XU$ with respect to the cylinder object $\Delta^1$. Therefore $X$ takes naive homotopy equivalences in $\smk$ to weak equivalences in $\sSet$. The assertion is obtained from the fact that the affine line $\AA^1$ is naive homotopy equivalent to the point $\Spec(k)$ in $\smk$ where a homotopy equivalence is given by the map $k[X]\to k[X,Y]$, $X\mapsto XY$ of $k$-algebras.
\end{proof}

\begin{corollary}
Let $\SPre$ be equipped with a simplicial model structure in which every object of $\smk$ is cofibrant. Then the class
$$
C=\{U\times\AA^1\xrightarrow{pr}U\mid U\in\smk\}
$$
consists of weak equivalences. 
\end{corollary}
\begin{proof}
Lemma \ref{homotopyinvariant} provides that $\sSet(U,X)\to\sSet(U\times\AA^1,X)$ is a weak equivalence of simplicial sets for every enriched simplicial presheaf $X$ by an enriched version of the Yoneda Lemma. Weak equivalences in a simplicial model category are detected by the property of the above morphism being a weak equivalence of simplicial sets for all fibrant objects $X$ \cite[Corollary 9.7.5]{hirsch}.
\end{proof}

\begin{corollary}\label{fibrant}
Consider the adjunction $L:\sPre\adjoint \SPre:R$ and the class
$$
C=\{U\times\AA^1\xrightarrow{pr}U\mid U\in\smk\}
$$
of morphisms of simplicial presheaves. Let $\sPre$ be equipped with a Bousfield localized model structure $L_C(\sPre)$ in which every object of $\smk$ is cofibrant. Suppose that the $R$-lifted model structure on $\SPre$ exists. Then an object $X$ of $\SPre$ is fibrant if and only if the object $R(X)$ is fibrant in $\sPre$ before localizing. 
\end{corollary}

\begin{lemma}
Consider the adjunction $L:\sPre\adjoint \SPre:R$ and let $\sPre$ be equipped with a left proper cofibrantly generated model structure with $\AA^1$-local weak equivalences as weak equivalences and with the property that every cofibration is in particular a monomorphism. If the category of enriched simplicial presheaves is endowed with the $R$-lifted model structure, then it is a left proper model category.
\end{lemma}
\begin{proof}
It is sufficient to show that the $R$-lifted $\affin^{1}$-local injective model structure is left proper. The injective model structure on $\SPre$ is left proper and it is the $R$-lifted model of the injective structure on $\sPre$ \cite[Proposition B.1]{HTT}. Let $B$ be a class of cofibrations in $\sPre$, such that the localization at $B$ is the local injective model structure. Then $(L,R)$ is a Quillen adjunction between the local injective model on $\sPre$ and the localization $M$ of the injective model structure on $\SPre$ at $L(B)$ \cite[Theorem 3.3.20]{hirsch}. We show that $M$ coincides with the $R$-lifted $\affin^1$-local injective model structure on $\SPre$. Let the injective model structure on $\sPre$ be $(I,J)$-cofibrantly generated, then the injective model structure on $\SPre$ is $(LI,LJ)$-cofibrantly generated and so is its left Bousfield localization $M$. By the same arguments, the $R$-lifted $\affin^1$-local injective model structure on $\SPre$ is also $(LI,LJ)$-cofibrantly generated. Hence both model structures have the same cofibrations. Moreover, their fibrant objects coincide by Corollary \ref{fibrant} and the fact that an object $X$ is fibrant in the Bousfield localization $M$ if and only if $\sSet(-,X)$ maps $B$ to weak equivalences. Therefore the model structures are the same since a model structure is determined by its cofibrations and its fibrant objects.
\end{proof}

\begin{remark}
The previous statements might suggest that it is possible to get a model for the motivic homotopy category by lifting a local model structure to the category of enriched simplicial presheaves. In view of Lemma \ref{liftinglemma} one observes that a $(I,J)$-cofibrantly generated model structure lifts via $(L,R)$ to the category of enriched simplicial presheaves if $\Sing(j)$ is a local weak equivalence for every generating acyclic cofibration $j$ in $J$, but the singular functor does not preserve local weak equivalences in general.
\end{remark}

\bibliographystyle{amsalpha}
\bibliography{ref.bib}

\providecommand{\bysame}{\leavevmode\hbox to3em{\hrulefill}\thinspace}
\providecommand{\MR}{\relax\ifhmode\unskip\space\fi MR }
\providecommand{\MRhref}[2]{%
  \href{http://www.ams.org/mathscinet-getitem?mr=#1}{#2}
}
\providecommand{\href}[2]{#2}
\begin{thebibliography}{Bor94b}

\bibitem[Bor94a]{borc94}
Francis Borceaux, \emph{Handbook of categorial algebra 1 - basic category
  theory}, Encyclopedia of Mathematics and its Applications, vol.~50, Cambridge
  University Press, 1994.

\bibitem[Bor94b]{borc942}
\bysame, \emph{Handbook of categorial algebra 2 - categories and structures},
  Encyclopedia of Mathematics and its Applications, vol.~51, Cambridge
  University Press, 1994.

\bibitem[Day70]{day}
Brian Day, \emph{On closed categories of functors}, Reports of the {M}idwest
  {C}ategory {S}eminar, {IV}, Lecture Notes in Mathematics, Vol. 137, Springer,
  Berlin, 1970, pp.~1--38.

\bibitem[GJ99]{goerss}
Paul~Gregory Goerss and John~Frederick Jardine, \emph{Simplicial homotopy
  theory}, Progress in Mathematics, vol. 174, Birkh{\"a}user, 1999.

\bibitem[Hir03]{hirsch}
Philip~Steven Hirschhorn, \emph{Model categories and their localizations},
  Mathematical Surveys and Monographs, American Mathematical Society, 2003.

\bibitem[Lur09]{HTT}
Jacob Lurie, \emph{{Higher Topos Theory}}, Annals of Mathematics Studies, vol.
  170, Princeton University Press, Princeton, NJ, 2009.

\bibitem[Mor04]{morel2004motivic}
Fabien Morel, \emph{{On the motivic $\pi_0$ of the sphere spectrum}},
  Axiomatic, enriched and motivic homotopy theory (2004), 219--260.

\bibitem[MV99]{mv99}
Fabien Morel and Vladimir Voevodsky, \emph{{$A^1$}-homotopy theory of schemes},
  Publications Math{\'e}matiques de l'Institut des Hautes {\'E}tudes
  Scientifiques \textbf{90} (1999), 45--143.

\end{thebibliography}

\end{document}